\documentclass[twoside,12pt]{article} 
\usepackage[margin=1.0in]{geometry}

\usepackage{amssymb,amsmath,amsthm}
\usepackage{mathtools}
\usepackage{bm,lscape}
\usepackage{bbm}
\usepackage{mathrsfs,dsfont}
\usepackage{wasysym}
\RequirePackage[
    colorlinks=true,
    linkcolor=blue,
    urlcolor=blue,
    citecolor=blue]{hyperref}
\RequirePackage{natbib}
\usepackage[]{algorithm2e}
\usepackage[utf8]{inputenc}
\usepackage[english]{babel}

\usepackage{graphicx}
\usepackage{color}
\usepackage{rotating}
\usepackage{authblk}
\usepackage{appendix}

\allowdisplaybreaks

\usepackage{booktabs}

\newtheorem{theorem}{Theorem}

\newtheorem{remark}[theorem]{Remark}

\providecommand{\keywords}[1]
{
  \small	
  \textbf{\textit{Keywords:}} #1
}

\begin{document}

\title{A compound Poisson perspective of Ewens-Pitman sampling model}

% \author{
% Emanuele Dolera\\
% \texttt{emanuele.dolera@unipv.it}\\
% Department of Mathematics\\ 
% University of Pavia, Italy\\
% \and
% Stefano Favaro\\
% \texttt{stefano.favaro@unito.it}\\
% Department of Economics and Statistics\\ 
% University of Torino and Collegio Carlo Alberto, Italy\\
% \and
% Stefano Peluchetti\\
% \texttt{speluchetti@cogent.co.jp}\\
% Cogent Labs, Tokyo, Japan
% }

\author[1]{Emanuele Dolera\thanks{emanuele.dolera@unipv.it}}
\author[2]{Stefano Favaro \thanks{stefano.favaro@unito.it}}
\affil[1]{\small{Department of Mathematics, University of Pavia, Italy}}
\affil[2]{\small{Department of Economics and Statistics, University of Torino and Collegio Carlo Alberto, Italy}}

\maketitle

\begin{abstract}
The Ewens-Pitman sampling model (EP-SM) is a distribution for random partitions of the set $\{1,\ldots,n\}$, with $n\in\mathbb{N}$, which is index by real parameters $\alpha$ and $\theta$ such that either $\alpha\in[0,1)$ and $\theta>-\alpha$, or $\alpha<0$ and $\theta=-m\alpha$ for some $m\in\mathbb{N}$. For $\alpha=0$ the EP-SM reduces to the celebrated Ewens sampling model (E-SM), which admits a well-known compound Poisson perspective in terms of the log-series compound Poisson sampling model (LS-CPSM). In this paper, we consider a generalization of the LS-CPSM, which is referred to as the negative Binomial compound Poisson sampling model (NB-CPSM), and we show that it leads to extend the compound Poisson perspective of the E-SM to the more general EP-SM for either $\alpha\in(0,1)$, or $\alpha<0$. The interplay between the NB-CPSM and the EP-SM is then applied to the study of the large $n$ asymptotic behaviour of the number of blocks in the corresponding random partitions, leading to a new proof of Pitman's $\alpha$ diversity. We discuss the proposed results, and conjecture that analogous compound Poisson representations may hold for the class of $\alpha$-stable Poisson-Kingman sampling models, of which the EP-SM is a noteworthy special case. 
\end{abstract}


%%%%%%%%%%%%%%%%%%%%%%%%%%%%%%%%
%%%%%%%%%%%%%%%%%%%%%%%%%%%%%%%%
%%%%%%%%%%%%%%%%%%%%%%%%%%%%%%%%
%%%%%%%%%%%%%%%%%%%%%%%%%%%%%%%%

\section{Introduction}

The Pitman-Yor process is a discrete random probability measure indexed by real parameters $\alpha$ and $\theta$ such that either $\alpha\in[0,1)$ and $\theta>-\alpha$, or $\alpha<0$ and $\theta=-m\alpha$ for some $m\in\mathbb{N}$. See, e.g., the works of \citet{Per(92)},  \citet{Pit(95)} and \citet{Pit(97)}. Let $\{V_i\}_{i\geq1}$ be independent random variables such that $V_i$ is distributed as a Beta distribution with parameter $(1-\alpha,\theta+i\alpha)$, for $i\geq1$, with the convention for $\alpha<0$ that $V_{m}=1$ and $V_{i}$ is undefined for $i>m$. If $P_1 := V_1$ and $P_i := V_i\prod_{1\leq j\leq i-1}(1-V_j)$ for $i \geq 2$, such that $\sum_{i\geq1}P_i=1$ almost surely, then the Pitman-Yor process is the random probability measure $\tilde{\mathfrak{p}}_{\alpha,\theta}$ on $(\mathbb{N},2^{\mathbb{N}})$ such that $\tilde{\mathfrak{p}}_{\alpha,\theta}(\{i\}) = P_i$ for $i\geq1$. The Dirichlet process \citep{Fer(73)} arises for $\alpha=0$. Because of the discreteness of $\tilde{\mathfrak{p}}_{\alpha,\theta}$, a random sample $(X_1,\ldots,X_n)$ induces a random partition $\Pi_{n}$ of $\{1,\ldots,n\}$ by means of the equivalence $i\sim j\iff X_i = X_j$ \citep{Pit(06)}. Let $K_{n}(\alpha,\theta):=K_{n}(X_{1},\ldots,X_{n})\leq n$ be the number of blocks of $\Pi_{n}$ and let $M_{r,n}(\alpha,\theta):=M_{r,n}(X_{1},\ldots,X_{n})$, for $r=1,\ldots,n$, be the number of blocks with frequency $r$ of $\Pi_{n}$ with $\sum_{1\leq r\leq n}M_{r,n}=K_{n}$ and $\sum_{1\leq r\leq n}rM_{r,n}=n$. \citet{Pit(95)} showed that
\begin{align}\label{eq_ewe_mod}
&\text{Pr}[(M_{1,n}(\alpha,\theta),\ldots,M_{n,n}(\alpha,\theta))=(x_{1},\ldots,x_{n})]=n!\frac{\left(\frac{\theta}{\alpha}\right)_{(\sum_{i=1}^{n}x_{i})}}{(\theta)_{(n)}}\prod_{i=1}^{n}\frac{\left(\frac{\alpha(1-\alpha)_{(i-1)}}{i!}\right)^{x_{i}}}{x_{i}!},
\end{align}
with $(x)_{(n)}$ being the ascending factorial of $x$ of order $n$, i.e. $(x)_{(n)}:=\prod_{0\leq i\leq n-1 }(x+i)$. The distribution \eqref{eq_ewe_mod} is referred to as Ewens-Pitman sampling model (EP-SM), and for $\alpha=0$ it reduces to the celebrated Ewens sampling model (E-SM) introduced by \citet{Ewe(72)}. The Pitman-Yor process plays a critical role in a variety of research areas, such as mathematical population genetics, Bayesian nonparametric statistics, statistical machine learning, excursion theory, combinatorics and statistical physics. See \citet{Pit(06)} and \citet{Cra(16)} for a comprehensive treatment of this subject.

The E-SM admits a well-known compound Poisson perspective in terms of the log-series compound Poisson sampling model (LS-CPSM). See \citet{Cha(07)}, and references therein, for an overview on compound Poisson models and generalizations thereof. In particular, we consider a generic population of individuals with a random number $K$ of distinct types, and let $K$ be distributed according to a Poisson distribution with parameter $\lambda=-z\log(1-q)$ for $q\in(0,1)$ and $z>0$. For $i\in\mathbb{N}$ let $N_{i}$ denote the random number of individuals of type $i$ in the population, and let the $N_{i}$'s to be independent of $K$ and independent each other, with the same distribution of the form
\begin{equation}\label{eq:logser}
\text{Pr}[N_{1}=x]=-\frac{1}{x\log(1-q)}q^{x}
\end{equation}
for $x\in\mathbb{N}$. Let $S=\sum_{1\leq i\leq K}N_{i}$ and let $M_{r}=\sum_{1\leq i\leq K}\mathbbm{1}_{\{N_{i}=r\}}$ for $r=1,\ldots,S$, that is $M_{r}$ is the random number of $N_{i}$'s equal to $r$ such that $\sum_{r\geq1}M_{r}=K$ and $\sum_{r\geq1}rM_{r}=S$. If $(M_{1}(z,n),\ldots,M_{n}(z,n))$ denotes a random variable whose distribution coincides with the conditional distribution of $(M_{1},\ldots,M_{S})$ given $S=n$, then \citep[Section 3]{Cha(07)} it holds that
\begin{align}\label{eq:sampling_0}
&\text{Pr}[(M_{1}(z,n),\ldots,M_{n}(z,n))=(x_{1},\ldots,x_{n})]=\frac{n!}{(z)_{(n)}}\prod_{i=1}^{n}\frac{\left(\frac{z}{i}\right)^{x_{i}}}{x_{i}!}.
\end{align}
The distribution \eqref{eq:sampling_0} is referred to as the LS-CPSM, and it equivalent to the E-SM. That is, the distribution displayed in \eqref{eq:sampling_0} coincides with the distribution \eqref{eq_ewe_mod} with $\alpha=0$. Therefore, the distributions of $K(z,n)=\sum_{1\leq r\leq n}M_{r}(z,n)$ and $M_{r}(z,n)$ coincide with the distributions of $K_{n}(0,z)$ and $M_{r,n}(0,z)$, respectively. Let $\stackrel{w}{\longrightarrow}$ denote the weak convergence for random variables. From the work of  \citet{Kor(73)}, $K(z,n)/\log n\stackrel{w}{\longrightarrow}z$ as $n\rightarrow+\infty$, whereas from \citet{Ewe(72)} it follows that $M_{r}(z,n)\stackrel{w}{\longrightarrow}P_{z/r}$ as $n\rightarrow+\infty$, where $P_{z}$ is a Poisson random variable with parameter $z$. 

In this paper, we consider a generalization of the LS-CPSM, which is referred to as the negative Binomial compound Poisson sampling model (NB-CPSM). In particular, the NB-CPSM is indexed by a pair of real parameters $\alpha$ and $z$ such that either $\alpha\in(0,1)$ and $z>0$, or $\alpha<0$ and $z<0$. The LS-CPSM is recovered by letting $\alpha\rightarrow0$ and $z>0$. We show that the NB-CPSM leads to extend the compound Poisson perspective of the E-SM to the more general EP-SM for either $\alpha\in(0,1)$, or $\alpha<0$. That is, we show that: i) for $\alpha\in(0,1)$ the EP-SM \eqref{eq_ewe_mod} admits a representation as a randomized NB-CPSM with $\alpha\in(0,1)$ and $z>0$, where the randomization acts on $z$ with respect a scale mixture between a Gamma and a scaled Mittag-Leffler distribution \citep{Pit(06)}; ii) for $\alpha<0$ the NB-CPSM admits a representation in terms of a randomized EP-SM with $\alpha<0$ and $\theta=-m\alpha$ for some $m\in\mathbb{N}$, where the randomization acts on $m$ with respect to a tilted Poisson distribution arising from the Wright function \citep{Wri(35)}. The interplay between the NB-CPSM and the EP-SM is then applied to the study of the large $n$ asymptotic behaviour of the number of distinct blocks in the random partitions induced by the corresponding sampling models. In particular, by combining the randomized representation in i) with the large $n$ asymptotic behaviour or the number of distinct blocks under the NB-CPSM, we present a new proof of Pitman's $\alpha$-diversity \citep{Pit(06)}, namely the large $n$ asymptotic behaviour of $K_{n}(\alpha,\theta)$ under the EP-SM.

%%%%%%%%%%%%%%%%%%%%%%%%%%%%%%%%
%%%%%%%%%%%%%%%%%%%%%%%%%%%%%%%%
%%%%%%%%%%%%%%%%%%%%%%%%%%%%%%%%
%%%%%%%%%%%%%%%%%%%%%%%%%%%%%%%%

\section{A compound Poisson perspective of EP-SM}\label{sec21}

We start by introducing the NB-CPSM and investigating the large $n$ asymptotic behaviour of some statistics of its induced random partition. To introduce the NB-CPSM, we consider a generic population of individuals with a random number $K$ of types, and let $K$ be distributed as a Poisson distribution with parameter $\lambda=z[1-(1-q)^{\alpha}]$ such that either $q\in(0,1)$, $\alpha\in(0,1)$ and $z>0$, or $q\in(0,1)$, $\alpha<0$ and $z<0$. For $i\in\mathbb{N}$ let $N_{i}$ be the random number of individuals of type $i$ in the population, and let the $N_{i}$'s to be independent of $K$ and independent each other, with the same distribution
\begin{equation}\label{eq:negbin}
\text{Pr}[N_{1}=x]=-\frac{1}{[1-(1-q)^{\alpha}]}{\alpha\choose x}(-q)^{x}
\end{equation}
for $x\in\mathbb{N}$. Let $S=\sum_{1\leq i\leq K}N_{i}$ and $M_{r}=\sum_{1\leq i\leq K}\mathbbm{1}_{\{N_{i}=r\}}$ for $r=1,\ldots,S$, that is $M_{r}$ is the random number of $N_{i}$'s equal to $r$ such that $\sum_{r\geq1}M_{r}=K$ and $\sum_{r\geq1}rM_{r}=S$. If $(M_{1}(\alpha,z,n),\ldots,M_{n}(\alpha,z,n))$ is a random variable whose distribution coincides with the conditional distribution of $(M_{1},\ldots,M_{S})$ given $S=n$, then it holds \citep[Section 3]{Cha(07)} that
\begin{align}\label{eq:sampling_alpha}
&\text{Pr}[(M_{1}(\alpha,z,n),\ldots,M_{n}(\alpha,z,n))=(x_{1},\ldots,x_{n})]=\frac{n!}{\sum_{j=0}^{n}\mathscr{C}(n,j;\alpha)z^{j}}\prod_{i=1}^{n}\frac{\left[z\frac{\alpha(1-\alpha)_{(i-1)}}{i!}\right]^{x_{i}}}{x_{i}!},
\end{align}
where $\mathscr{C}(n,j;\alpha)=\frac{1}{j!}\sum_{0\leq i\leq j}{j\choose i}(-1)^{i}(-i\alpha)_{(n)}$ is the generalized factorial coefficient \citep{Cha(05)}, with the proviso $\mathscr{C}(n,0,\alpha)=0$ for all $n\in\mathbb{N}$, $\mathscr{C}(n,j,\alpha)=0$ for all $j>n$ and $\mathscr{C}(0,0,\alpha)=1$. The distribution \eqref{eq:sampling_alpha} is referred to as the NB-CPSM. In particular, as $\alpha\rightarrow0$, it is easy to show that the distribution \eqref{eq:negbin} reduces to the distribution \eqref{eq:logser}. Accordingly, as $\alpha\rightarrow0$, the NB-CPSM \eqref{eq:sampling_alpha} reduces to the LS-CPSM \eqref{eq:sampling_0}. The next theorem states the large $n$ asymptotic behaviour of the counting statistics $K(\alpha,z,n)=\sum_{1\leq r\leq n}M_{r}(\alpha,z,n)$ and $M_{r}(\alpha,z,n)$ arising from the NB-CPSM.

\begin{theorem}\label{teo1} Let $P_{\lambda}$ denote a Poisson random variable with parameter $\lambda>0$. As $n\rightarrow+\infty$ it holds 
\begin{itemize}
\item[i)] for $\alpha\in(0,1)$ and $z>0$
\begin{equation}\label{cp_k_pos}
K(\alpha,z,n)\stackrel{w}{\longrightarrow} 1+P_{z}
\end{equation}
and
\begin{equation}\label{cp_m_pos}
M_{r}(\alpha,z,n)\stackrel{w}{\longrightarrow}P_{\frac{\alpha(1-\alpha)_{(r-1)}}{r!}z};
\end{equation}
\item[ii)] for $\alpha<0$ and $z<0$
\begin{equation}\label{cp_k_neg}
\frac{K(\alpha,z,n)}{n^{\frac{-\alpha}{1-\alpha}}}\stackrel{w}{\longrightarrow} \frac{(\alpha z)^{\frac{1}{1-\alpha}}}{-\alpha}
\end{equation}
and
\begin{equation}\label{cp_m_neg}
M_{r}(\alpha,z,n)\stackrel{w}{\longrightarrow} P_{\frac{\alpha(1-\alpha)_{(r-1)}}{r!}z}.
\end{equation}
\end{itemize}
\end{theorem}

\begin{proof}
As regard the proof of the asymptotic behaviour in \eqref{cp_k_pos}, we start by recalling that the probability generating function $G(\cdot;\lambda)$ of $P_{\lambda}$ is $G(s;\lambda)=\exp\{-\lambda(s-1)\}$ for any $s>0$. Now, let $G(\cdot; \alpha,z,n)$ be the probability generating function of $K(\alpha,z,n)$. The distribution of $K(\alpha,z,n)$ follows by combining the NB-CPSM \eqref{eq:sampling_alpha} with Theorem 2.15 of  \citet{Cha(05)}. In particular, it follows that
\begin{displaymath}
G(s; \alpha,z,n)=\frac{\sum_{j=1}^{n}\mathscr{C}(n,j;\alpha)(sz)^{j}}{\sum_{j=1}^{n}\mathscr{C}(n,j;\alpha)z^{j}}.
\end{displaymath}
Hereafter, we show that $G(s; \alpha,z,n) \rightarrow s\exp\{z(s-1)\}$ as $n\rightarrow+\infty$, for any $s > 0$, which implies \eqref{cp_k_pos}. In particular, by a direct application of the definition of $\mathscr{C}(n,k;\alpha)$ we write the following identities 
\begin{align*}
\sum_{j=1}^n \mathscr{C}(n,j;\alpha) z^j&= \sum_{i=1}^n (-1)^i (-i\alpha)_{(n)} \sum_{k=i}^n \frac{1}{k!} {k\choose i} z^k= \sum_{i=1}^n (-1)^i (-i\alpha)_{(n)} e^z z^i \frac{\Gamma(n-i+1,z)}{i!\Gamma(n-i+1)},
\end{align*}
where $\Gamma(a,x) := \int_x^{+\infty} t^{a-1} e^{-t} \text{d} t$ denotes the incomplete gamma function for $a,x > 0$ and $\Gamma(a) := \int_{0}^{+\infty} t^{a-1} e^{-t} \text{d} t$ denotes the Gamma function for $a>0$. Accordingly, we can write the following identity
\begin{displaymath}
G(s; \alpha,z,n) = e^{z(s-1)} \frac{-zs \frac{\Gamma(n,zs)}{\Gamma(n)} + \sum_{i=2}^n (-1)^i \frac{(-i\alpha)_{(n)}}{(-\alpha)_{(n)}} (zs)^i \frac{\Gamma(n-i+1,zs)}{i!\Gamma(n-i+1)}}
{-z \frac{\Gamma(n,z)}{\Gamma(n)} + \sum_{i=2}^{n} (-1)^i \frac{(-i\alpha)_{(n)}}{(-\alpha)_{(n)}} z^i \frac{\Gamma(n-i+1,z)}{i!\Gamma(n-i+1)}}\ .
\end{displaymath}
Now, since $\lim_{n \rightarrow +\infty} \frac{\Gamma(n,x)}{\Gamma(n)} = 1$ for any $x > 0$, the proof \eqref{cp_k_pos} is completed by showing that, for any $t > 0$,
\begin{equation}\label{main_to_prove}
\lim_{n \rightarrow +\infty} \sum_{i=2}^n (-1)^i \frac{(-i\alpha)_{(n)}}{(-\alpha)_{(n)}} \frac{\Gamma(n-i+1,t)}{\Gamma(n-i+1)} \frac{t^i}{i!} = 0.
\end{equation}
By the definition of ascending factorials and the reflection formula of the Gamma function, it holds true
\begin{displaymath}
\frac{(-i\alpha)_{(n)}}{(-\alpha)_{(n)}} = \frac{\Gamma(n-i\alpha)}{\Gamma(n-\alpha)} \frac{\sin i\pi\alpha}{\pi} \Gamma(i\alpha + 1) \Gamma(-\alpha). 
\end{displaymath}
In particular, by means of the monotonicity of the function $[1,+\infty) \ni z \mapsto \Gamma(z)$, we can write the identity
\begin{equation}\label{first_bound}
\frac{1}{i!} \Big{|}\frac{(-i\alpha)_{(n)}}{(-\alpha)_{(n)}}\Big{|} \leq \frac{|\Gamma(-\alpha)|}{\pi} \frac{\Gamma(n-2\alpha)}{\Gamma(n-\alpha)}\frac{\Gamma(i\alpha + 1)}{i!}  
\end{equation}
for any $n \in \mathbb{N}$ such that $n > 1/(1-\alpha)$, and $i \in \{2, \dots, n\}$. Note that $\frac{\Gamma(n,x)}{\Gamma(n)} \leq 1$. 
Then we apply \eqref{first_bound} to get
\begin{align*}
\Big{|} \sum_{i=2}^n (-1)^i \frac{(-i\alpha)_{(n)}}{(-\alpha)_{(n)}} \frac{\Gamma(n-i+1,t)}{\Gamma(n-i+1)} \frac{t^i}{i!}\Big{|}&\leq \sum_{i=2}^n \frac{t^i}{i!} 
\Big{|}\frac{(-i\alpha)_{(n)}}{(-\alpha)_{(n)}}\Big{|}\\
& \leq \frac{|\Gamma(-\alpha)|}{\pi} \frac{\Gamma(n-2\alpha)}{\Gamma(n-\alpha)} \sum_{i\geq0} t^i \frac{\Gamma(i\alpha + 1)}{i!}.
\end{align*}
Now, by means of Stirling approximation it holds $\frac{\Gamma(n-2\alpha)}{\Gamma(n-\alpha)} \sim \frac{1}{n^{\alpha}}$ as $n\rightarrow+\infty$. Moreover, we have that
\begin{displaymath}
\sum_{i\geq0} t^i \frac{\Gamma(i\alpha + 1)}{i!} = \int_0^{+\infty} e^{tz^{\alpha} - z} \text{d} z < +\infty
\end{displaymath}
where the finiteness of the integral follows, for any fixed $t > 0$, from the fact that $tz^{\alpha} < \frac 12 z$ if $z > (2t)^{\frac{1}{1-\alpha}}$. This completes the proof of \eqref{main_to_prove}, and hence the proof of \eqref{cp_k_pos}. As regard the proof of \eqref{cp_m_pos}, we make use of the falling factorial moments of $M_{r}(\alpha,z,n)$, which follows by combining the NB-CPSM \eqref{eq:sampling_alpha} with Theorem 2.15 of \citet{Cha(05)}. Let $(a)_{[n]}$ be the falling factorial of $a$ of order $n$, i.e. $(a)_{[n]}=\prod_{0\leq i\leq n-1}(a-i)$, for any $a\in\mathbb{R}^{+}$ and $n\in\mathbb{N}_{0}$ with the proviso $(a)_{[0]}=1$. Then, we write
\begin{align*}
&\mathbb{E}[(M_{r}(\alpha,z,n))_{[s]}]\\
&\quad=(-1)^{rs}(n)_{[rs]}{\alpha\choose r}^{s}(-z)^{s}\frac{\sum_{j=0}^{n-rs}\mathscr{C}(n-rs,j;\alpha)z^{j}}{\sum_{j=0}^{n}\mathscr{C}(n,j;\alpha)z^{j}}\\
&\quad=(-1)^{rs}(n)_{[rs]}{\alpha\choose r}^{s}(-z)^{s}\frac{(-z)\frac{\Gamma(n-rs,z)}{\Gamma(n-rs)}+\sum_{i=2}^{n-rs}(-1)^{i}\frac{(-i\alpha)_{(n-rs)}}{(-\alpha)_{(n-rs)}}(z)^{i}\frac{\Gamma(n-rs-i+1,z)}{i!\Gamma(n-rs-i+1)}}{(-z)\frac{\Gamma(n,z)}{\Gamma(n)}+\sum_{i=2}^{n}(-1)^{i}\frac{(-i\alpha)_{(n)}}{(-\alpha)_{(n)}}(z)^{i}\frac{\Gamma(n-i+1,z)}{\Gamma(n-i+1)}}\\
&\quad=(-1)^{rs}(n)_{[rs]}{\alpha\choose r}^{s}(-z)^{s}\\
&\quad\quad\times\frac{(-\alpha)_{(n-rs)}}{(-\alpha)_{(n)}}\frac{(-z)\frac{\Gamma(n-rs,z)}{\Gamma(n-rs)}+\sum_{i=2}^{n-rs}(-1)^{i}\frac{(-i\alpha)_{(n-rs)}}{(-\alpha)_{(n-lr)}}(z)^{i}\frac{\Gamma(n-rs-i+1,z)}{i!\Gamma(n-rs-i+1)}}{(-z)\frac{\Gamma(n,z)}{\Gamma(n)}+\sum_{i=2}^{n}(-1)^{i}\frac{(-i\alpha)_{(n)}}{(-\alpha)_{(n)}}(z)^{i}\frac{\Gamma(n-i+1,z)}{\Gamma(n-i+1)}}.
\end{align*}
Now, by means of the same argument applied in the proof of the statement in \eqref{cp_k_pos}, it holds true that
\begin{displaymath}
\lim_{n\rightarrow+\infty}\frac{(-z)\frac{\Gamma(n-rs,z)}{\Gamma(n-rs)}+\sum_{i=2}^{n-rs}(-1)^{i}\frac{(-i\alpha)_{(n-rs)}}{(-\alpha)_{(n-lr)}}(z)^{i}\frac{\Gamma(n-rs-i+1,z)}{i!\Gamma(n-rs-i+1)}}{(-z)\frac{\Gamma(n,z)}{\Gamma(n)}+\sum_{i=2}^{n}(-1)^{i}\frac{(-i\alpha)_{(n)}}{(-\alpha)_{(n)}}(z)^{i}\frac{\Gamma(n-i+1,z)}{\Gamma(n-i+1)}}=1.
\end{displaymath}
Then,
\begin{displaymath}
\lim_{n\rightarrow+\infty}\mathbb{E}[(M_{r}(\alpha,z,n))_{[s]}]=(-1)^{rs}{\alpha\choose r}^{s}(-z)^{s}=\left[\frac{\alpha(1-\alpha)_{(r-1)}}{r!}z\right]^{s}
\end{displaymath}
follows from the fact that $(n)_{[rs]}\sim\frac{(-\alpha)_{(n-rs)}}{(-\alpha)_{(n)}}$ as $n\rightarrow+\infty$. Finally, the proof of the large $n$ asymptotics \eqref{cp_m_pos} is completed by recalling that falling factorial moment of order $s$ of $P_{\lambda}$ is $\mathbb{E}[(P_{\lambda})_{[s]}]=\lambda^{s}$. 

As regard the proof of the asymptotic behaviour in \eqref{cp_k_neg}, let  $\alpha = -\sigma$ for any $\sigma > 0$ and let $z = -\zeta$ for any $\zeta > 0$. Then, by a direct application of Equation 2.27 of \citet{Cha(05)}, we write the identity
\begin{displaymath}
\sum_{j=0}^{n}\mathscr{C}(n,j;-\sigma)(-\zeta)^{j}=(-1)^{n}\sum_{v=0}^{n}s(n,v)(-\sigma)^{v}\sum_{j=0}^{v}\zeta^{j}S(v,j),
\end{displaymath}
where $S(v,j)$ is the Stirling number of that second type. Now, note that $\sum_{0\leq j\leq v}^{v}\zeta^{j}S(v,j)$ is the moment of order $v$ of a Poisson random variable with parameter $\zeta>0$. Then, we write the following identities 
\begin{align}\label{exp_gen_fact}
\sum_{j=0}^{n}\mathscr{C}(n,j;-\sigma)(-\zeta)^{j}&=\sum_{v=0}^{n}|s(n,v)|\sigma^{v}\sum_{j\geq0}j^{v}\text{e}^{-\zeta}\frac{\zeta^{j}}{j!}=\sum_{j\geq0}\text{e}^{-\zeta}\frac{\zeta^{j}}{j!}\int_{0}^{+\infty}x^{n}f_{G_{\sigma j,1}}(x)\text{d} x.
\end{align}
That is, 
\begin{equation} \label{bernardoNEW}
B_n(w) = \mathbb{E}[(G_{\sigma P_{w},1})^n],
\end{equation}
where $G_{a,1}$ and $P_w$ denote two independent random variables such that $G_{a,1}$ is a Gamma random variable with shape parameter $a > 0$ and scale parameter $1$, and $P_w$ is a Poisson random variable with parameter $w$. Accordingly, the distribution of the random variable $G_{\sigma P_{w},1}$, say $\mu_{\sigma,w}$ is the following 
\begin{displaymath}
\mu_{\sigma,w}(\text{d} t) = e^{-w} \delta_0(\text{d} t) + \left(\sum_{j\geq1} \frac{e^{-w} w^j}{j!} \frac{1}{\Gamma(j\sigma)} e^{-t} t^{j\sigma - 1} \right) \text{d} t
\end{displaymath} 
for $t>0$. The discrete component of the distribution $\mu_{\sigma,w}$ does not contribute in the expectation \eqref{bernardoNEW}, so that we focus on the absolutely continuous component, whose density can be written as follows
\begin{displaymath} 
\sum_{j\geq1}\frac{e^{-w} w^j}{j!} \frac{1}{\Gamma(j\sigma)} e^{-t} t^{j\sigma - 1} = \frac{e^{-(w+t)}}{t} W_{\sigma, 0}(wt^{\sigma}), 
\end{displaymath}
where $W_{\sigma, \tau}(y) := \sum_{j\geq0} \frac{y^j}{j! \Gamma(j\sigma + \tau)}$ is the Wright function \citep{Wri(35)}. In particular, for $\tau = 0$ it holds
\begin{equation} \label{BnWright}
B_n(w) = \int_0^{+\infty} t^n \frac{e^{-(w+t)}}{t} W_{\sigma, 0}(wt^{\sigma}) \text{d} t\ . 
\end{equation}
If we split the integral as $\int_0^M + \int_M^{+\infty}$ for any $M>0$, the contribution of the latter integral is overwhelming with respect to the contribution of the former. Then $W_{\sigma, 0}$ can be equivalently replaced by the asymptotics $W_{\sigma, 0}(y) \sim c(\sigma) y^{\frac{1}{2(1+\sigma)}} \exp\{ \sigma^{-1}(\sigma+1) (\sigma y)^{\frac{1}{1+\sigma}}\}$, as $ y\rightarrow+\infty$, for some constant $c(\sigma)$ depending solely on $\sigma$. See Theorem 2 in \citet{Wri(35)}. Hence, we can write the identity
\begin{align*}
B_n(w) &\sim c(\sigma) \int_0^{+\infty} t^{n-1} e^{-(w+t)} (wt^{\sigma})^{\frac{1}{2(1+\sigma)}} \exp\left\{ \frac{\sigma+1}{\sigma} (\sigma wt^{\sigma})^{\frac{1}{1+\sigma}} \right\} \text{d} t \nonumber \\
&= c(\sigma) e^{-w} w^{\frac{1}{2(1+\sigma)}}  \int_0^{+\infty} t^{n + \frac{\sigma}{2(1+\sigma)} -1} \exp\{ A(w,\sigma) t^{\frac{\sigma}{1+\sigma}} - t \} \text{d} t,
\end{align*}
where $A(w,\sigma) := \frac{\sigma+1}{\sigma} (\sigma w)^{\frac{1}{1+\sigma}}$. Then, the problem is reduced to an integral whose asymptotic behaviour is described in \citet{Ber(58)}. From Equation 31 of \citet{Ber(58)} and Stirling approximation, we have
\begin{equation} \label{BnAsymptotic}
B_n(w) \sim c(\sigma) e^{-w} w^{\frac{1}{2(1+\sigma)}} \Gamma(n)\exp\left\{ A(w,\sigma) n^{\frac{\sigma}{1+\sigma}} \right\} \ . 
\end{equation}
In particular, observe that such a last asymptotic expansion leads directly to \eqref{cp_k_neg}. Indeed let $G(\cdot; -\sigma,-\zeta,n)$ be the probability generating function of the random variable $K(-\sigma,-\zeta,n)$, which reads as $G(s; -\sigma,-\zeta,n)= B_n(s\zeta)/B_n(\zeta)$ for $s>0$. Then, by means of \eqref{BnAsymptotic}, for any fixed $s>0$ we write
\begin{equation} \label{GnAsymptotic}
G(s; -\sigma,-\zeta,n) \sim e^{-w(s-1)} s^{\frac{1}{2(1+\sigma)}} \exp\left\{ n^{\frac{\sigma}{1+\sigma}} \frac{\sigma+1}{\sigma} (\sigma \zeta)^{\frac{1}{1+\sigma}} [s^{\frac{1}{1+\sigma}} - 1]\right\} \ .
\end{equation}
Note that \eqref{BnAsymptotic} holds uniformly in $w$ in a compact set. Accordingly, we consider the function $G(s; -\sigma,-\zeta,n)$ evaluated at some point $s_n$ and extend the validity of \eqref{GnAsymptotic} with $s_n$ in the place of $s$, as long as $\{s_n\}_{n\geq 1}$ varies in a compact subset of $[0,+\infty)$. Thus, we can choose $s_n = s^{\beta(n)}$ and $\beta(n) = \frac{1}{n^{\frac{\sigma}{1+\sigma}}}$ and notice that $\beta(n) \rightarrow 0$ as $n\rightarrow+\infty$. Thus, $s_n \simeq 1 + \beta(n)\log s \rightarrow 1$ and we have that
\begin{displaymath}
n^{\frac{\sigma}{1+\sigma}} \frac{\sigma+1}{\sigma} (\sigma w)^{\frac{1}{1+\sigma}} [s_n^{\frac{1}{1+\sigma}} - 1] \rightarrow \frac{(\sigma \zeta)^{\frac{1}{1+\sigma}}}{\sigma} \log s, 
\end{displaymath}
which implies that $K(-\sigma,-\zeta,n)\rightarrow\frac{(\sigma \zeta)^{\frac{1}{1+\sigma}}}{\sigma}$ as $n\rightarrow+\infty$. This completes the proof of \eqref{cp_k_neg}. As regard the proof \eqref{cp_m_neg}, let  $\alpha = -\sigma$ for any $\sigma > 0$ and let $z = -\zeta$ for any $\zeta > 0$. Similarly to the proof of \eqref{cp_m_pos}, here we make use of the falling factorial moments of $M_{r}(-\sigma,-\zeta,n)$. In particular, we can write
\begin{align*}
&\mathbb{E}[(M_{r}(-\sigma,\zeta,n))_{[s]}]=(-1)^{rs}(n)_{[rs]}{-\sigma\choose r}^{s}\zeta^{s}\frac{\sum_{j=0}^{n-rs}\mathscr{C}(n-rs,j;-\sigma)(-\zeta)^{j}}{\sum_{j=0}^{n}\mathscr{C}(n,j;-\sigma)(-\zeta)^{j}}.
\end{align*}
At this point, we can make use of the same large $n$ asymptotic arguments applied in the proof of statement \eqref{cp_m_pos}. In particular, by means of the large $n$ asymptotic \eqref{BnAsymptotic}, as $n\rightarrow+\infty$, it holds true that
\begin{displaymath}
\frac{\sum_{j=0}^{n-rs}\mathscr{C}(n-rs,j;-\sigma)(-\zeta)^{j}}{\sum_{j=0}^{n}\mathscr{C}(n,j;-\sigma)(-\zeta)^{j}}\sim n^{-rs}.
\end{displaymath}
Then,
\begin{displaymath}
\lim_{n\rightarrow+\infty}\mathbb{E}[(M_{r}(-\sigma,-\zeta,n))_{[s]}]=(-1)^{rs}{-\sigma\choose r}^{s}\zeta^{s}=\left[\frac{-\sigma(1+\sigma)_{(r-1)}}{r!}(-\zeta)\right]^{s}
\end{displaymath}
follows from the fact that $(n)_{[rs]}\sim n^{rs}$ as $n\rightarrow+\infty$. Finally, the proof of the large $n$ asymptotic behaviour in \eqref{cp_m_neg} is completed by recalling that falling factorial moment of order $s$ of $P_{\lambda}$ is $\mathbb{E}[(P_{\lambda})_{[s]}]=\lambda^{s}$. 
\end{proof}

In the rest of the present section, we make use of the NB-CPSM displayed in \eqref{eq:sampling_alpha} to introduce a compound Poisson perspective of the EP-SM. In particular, our main result extends the well-known compound Poisson perspective of the E-SM to the EP-SM for either $\alpha\in(0,1)$, or $\alpha<0$. For $\alpha\in(0,1)$ let $f_{\alpha}$ denote the density function of a positive $\alpha$-stable random variable $X_{\alpha}$, that is $X_{\alpha}$ is a random variable for which the moment generating function is $\mathbb{E}[\exp\{-tX_{\alpha}\}]=\exp\{-t^{\alpha}\}$ for any $t>0$. For $\alpha\in(0,1)$ and $\theta>-\alpha$ let  $S_{\alpha,\theta}$ be a positive random variable with density function
\begin{displaymath}
f_{S_{\alpha,\theta}}(s)=\frac{\Gamma(\theta+1)}{\alpha\Gamma(\theta/\alpha+1)}s^{\frac{\theta-1}{\alpha}-1}f_{\alpha}(s^{-\frac{1}{\alpha}}).
\end{displaymath}
That is, the random variable $S_{\alpha,\theta}$ is a scaled Mittag-Leffler random variable \citep[Chapter 1]{Pit(06)}. Now, let $G_{a,b}$ be a Gamma random variable with scale parameter $b>0$ and shape parameter $a>0$, and let assume that $G_{a,b}$ is independent of $S_{\alpha,\theta}$. Then, for $\alpha\in(0,1)$, $\theta>-\alpha$ and $n\in\mathbb{N}$ we define
\begin{equation}\label{distr_lat_1}
\bar{X}_{\alpha,\theta,n}\stackrel{d}{=}G^{\alpha}_{\theta+n,1}S_{\alpha,\theta}.
\end{equation}
Finally, for $\alpha<0$, $z<0$ and $n\in\mathbb{N}$ let $\tilde{X}_{\alpha,z,n}$ be a random variable on $\mathbb{N}$ whose distribution is a tilted Poisson distribution arising from the identity \eqref{exp_gen_fact}. Precisely, for any $x\in\mathbb{N}$ the distribution of $\tilde{X}_{\alpha,z,n}$ is
\begin{equation}\label{distr_lat_2}
\text{Pr}[\tilde{X}_{\alpha,z,n}=x]=\frac{1}{\sum_{j=1}^{n}\mathscr{C}(n,j;\alpha)z^{j}}\frac{\text{e}^{z}(-z)^{x}\Gamma(-x\alpha+n)}{x!\Gamma(-x\alpha)}.
\end{equation}
In the next theorem, we make use of the random variables $\bar{X}_{\alpha,\theta,n}$ and $\tilde{X}_{\alpha,z,n}$ to set an interplay between the NB-CPSM \eqref{eq:sampling_alpha} and the EP-SM \eqref{eq_ewe_mod}. This extends the compound Poisson perspective of the E-SM.

\begin{theorem}\label{teo3}
Let $(M_{1,n}(\alpha,\theta),\ldots,M_{n,n}(\alpha,\theta))$ be distributed as the EP-SM \eqref{eq_ewe_mod} and let $\bar{X}_{\alpha,\theta,n}$ be the random variable defined in \eqref{distr_lat_1}, which is independent of $(M_{1,n}(\alpha,\theta),\ldots,M_{n,n}(\alpha,\theta))$. Moreover, let $(M_{1}(\alpha,z,n),\ldots,M_{n}(\alpha,z,n))$ be distributed as the NB-CPSM \eqref{eq:sampling_alpha}, and let $\tilde{X}_{\alpha,z,n}$ be the random variable defined in \eqref{distr_lat_2}, which is independent of $(M_{1}(\alpha,z,n),\ldots,M_{n}(\alpha,z,n))$. Then, it holds true that
\begin{itemize}
\item[i)] for $\alpha\in(0,1)$ and $\theta>-\alpha$
\begin{displaymath}
(M_{1,n}(\alpha,\theta),\ldots,M_{n,n}(\alpha,\theta))\stackrel{d}{=}(M_{1}(\alpha,\bar{X}_{\alpha,\theta,n},n),\ldots,M_{n}(\alpha,\bar{X}_{\alpha,\theta,n},n));
\end{displaymath}
\item[ii)] for $\alpha<0$ and $z<0$
\begin{displaymath}
(M_{1}(\alpha,z,n),\ldots,M_{n}(\alpha,z,n))\stackrel{d}{=}(M_{1,n}(\alpha,-\tilde{X}_{\alpha,z,n}\alpha),\ldots,M_{n,n}(\alpha,-\tilde{X}_{\alpha,z,n}\alpha)).
\end{displaymath}
\end{itemize}
\end{theorem}

\begin{proof}
As regard the proof of statement i), its proof relies on the classical integral representation of the Gamma function. That is, by applying the integral representation of the function $\Gamma(\theta/\alpha+k)$ to the EP-SM \eqref{eq_ewe_mod}, for $x_{1},\ldots,x_{n} \in \{0,\dots, n\}$ with $\sum_{i=1}^n x_i = k$ and $\sum_{i=1}^n ix_i = n$, we can write that
\begin{align*}
&\text{Pr}[(M_{1,n}(\alpha,\theta),\ldots,M_{n,n}(\alpha,\theta))=(x_{1},\ldots,x_{n})]\\
&\quad=n!\frac{\alpha^{k}}{\Gamma(\theta + n)} \prod_{i=1}^{n}\frac{\left(\frac{(1-\alpha)_{(i-1)}}{i!}\right)^{x_{i}}}{x_{i}!}
\frac{\Gamma(\theta+1)}{\alpha\Gamma(\theta/\alpha+1)}\\
&\quad\quad\times\int_{0}^{+\infty}z^{\theta/\alpha-1}\text{e}^{-z}\frac{z^{k}}{\sum_{j=1}^{n}\mathscr{C}(n,j;\alpha)z^{j}}\left(\sum_{j=1}^{n}\mathscr{C}(n,j;\alpha)z^{j}\right)\text{d} z\\
&\text{[By Equation 13 of \citet{Fav(15)}]}\\
&\quad=n!\frac{\alpha^{k}}{\Gamma(\theta + n)}\prod_{i=1}^{n}\frac{\left(\frac{(1-\alpha)_{(i-1)}}{i!}\right)^{x_{i}}}{x_{i}!} \frac{\Gamma(\theta+1)}{\alpha\Gamma(\theta/\alpha+1)}\\
&\quad\quad\times\int_{0}^{+\infty}z^{\theta/\alpha-1}\text{e}^{-z}\frac{z^{k}}{\sum_{j=1}^{n}\mathscr{C}(n,j;\alpha)z^{j}}\left(\text{e}^{z}z^{n/\alpha}\int_{0}^{+\infty}y^{n}\text{e}^{-yz^{1/\alpha}}f_{\alpha}(y)\text{d} y\right)\text{d} z\\
&\quad=\int_{0}^{+\infty}\frac{n!}{\sum_{j=0}^{n}\mathscr{C}(n,j,\alpha)z^{j}}\prod_{i=1}^{n}\frac{\left(z\frac{\alpha(1-\alpha)_{(i-1)}}{i!}\right)^{x_{i}}}{x_{i}!}\\
&\quad\quad\times \frac{\Gamma(\theta+1)}{\alpha\Gamma(\theta + n) \Gamma(\theta/\alpha +1)} z^{\theta/\alpha+n/\alpha-1}
\int_{0}^{+\infty}y^{n}\text{e}^{-yz^{1/\alpha}}f_{\alpha}(y)\text{d} y\text{d} z\\
&\quad=\int_{0}^{+\infty}\text{Pr}[(M_{1}(\alpha,x,n),\ldots,M_{n}(\alpha,x,n))=(x_{1},\ldots,x_{n})]\\
&\quad\quad\times \frac{\Gamma(\theta+1)}{\alpha\Gamma(\theta + n) \Gamma(\theta/\alpha +1)}
z^{\theta/\alpha+n/\alpha-1}\int_{0}^{+\infty}y^{n}\text{e}^{-yz^{1/\alpha}}f_{\alpha}(y)\text{d}y\text{d} z\\
&\text{[By the distribution of $\bar{X}_{\alpha,\theta,n}$]}\\
&\quad=\int_{0}^{+\infty}\text{Pr}[(M_{1}(\alpha,z,n),\ldots,M_{n}(\alpha,z,n))=(x_{1},\ldots,x_{n})]f_{\bar{X}_{\alpha,\theta,n}}(z)\text{d}z,
\end{align*}
where $f_{\bar{X}_{\alpha,\theta,n}}$ is the density function of the random variable $\bar{X}_{\alpha,\theta,n}$. This completes the proof of i).

As regard the proof of statement ii), for any $\alpha<0$, $m\in\mathbb{N}$, $k\leq m$ and $n\in\mathbb{N}$ we start by defining the function $m\mapsto A(m;k,\alpha,n)=\frac{m!}{(m-k)!}\frac{\Gamma(-m\alpha)}{\Gamma(-m\alpha+n)}$, and then we consider the following identity
\begin{equation}\label{gamma_id_2}
\frac{(-z)^{k}}{\sum_{j=1}^{n}\mathscr{C}(n,j;\alpha)z^{j}}=\sum_{m\geq k}A(m;k,\alpha,n)\text{Pr}[\tilde{X}_{\alpha,z,n}=m].
\end{equation}
By applying \eqref{gamma_id_2} to the NB-CPSM \eqref{eq:sampling_alpha}, for $x_{1},\ldots,x_{n} \in \{0,\dots, n\}$ with $\sum_{i=1}^n x_i = k$ 
and $\sum_{i=1}^n ix_i = n$, we write
\begin{align*}
&\text{Pr}[(M_{1}(\alpha,z,n),\ldots,M_{n}(\alpha,z,n))=(x_{1},\ldots,x_{n})]\\
&\quad=\sum_{m\geq k}n!(-1)^{k}A(m;k,\alpha,n)\text{Pr}[\tilde{X}_{\alpha,z,n}=m]\prod_{i=1}^{n}\frac{\left(\frac{\alpha(1-\alpha)_{(i-1)}}{i!}\right)^{x_{i}}}{x_{i}!}\\
&\quad=\sum_{m\geq k}n!(-1)^{k}\frac{m!}{(m-k)!}\frac{\Gamma(-m\alpha)}{\Gamma(-m\alpha+n)}\text{Pr}[\tilde{X}_{\alpha,z,n}=m]\prod_{i=1}^{n}\frac{\left(\frac{\alpha(1-\alpha)_{(i-1)}}{i!}\right)^{x_{i}}}{x_{i}!}\\
&\quad=\sum_{m\geq k}n!\frac{\left(\frac{-m\alpha}{\alpha}\right)_{(k)}}{(-m\alpha)_{(n)}}\prod_{i=1}^{n}\frac{\left(\frac{\alpha(1-\alpha)_{(i-1)}}{i!}\right)^{x_{i}}}{x_{i}!}\text{Pr}[\tilde{X}_{\alpha,z,n}=m]\\
&\quad=\sum_{m\geq k}\text{Pr}[(M_{1}(\alpha,-m\alpha),\ldots,M_{n}(\alpha,-m\alpha))=(x_{1},\ldots,x_{n})]\text{Pr}[\tilde{X}_{\alpha,z,n}=m].
\end{align*}
This completes the proof of ii).
\end{proof}

Theorem \ref{teo3} presents a compound Poisson perspective of the EP-SM in terms of the NB-CPSM, thus extending the well-known compound Poisson perspective of the E-SM in terms of the LS-CPSM. Statement i) of Theorem \ref{teo3} shows that for $\alpha\in(0,1)$ and $\theta>-\alpha$ the EP-SM admits a representation in terms of the NB-CPSM with $\alpha\in(0,1)$ and $z>0$, where the randomization acts on the parameter $z$ with respect to the distribution \eqref{distr_lat_1}. Precisely, this is a compound mixed Poisson sampling model. That is, a compound sampling model in which the distribution of the random number $K$ of distinct types in the population is a mixture of Poisson distributions with respect to the law of $\bar{X}_{\alpha,\theta,n}$. Statement ii) of Theorem \ref{teo3} shows that for $\alpha<0$ and $z<0$ the NB-CPSM admits a representation in terms of a randomized EP-SM with $\alpha<0$ and $\theta=-m\alpha$ for some $m\in\mathbb{N}$, where the randomization acts on the parameter $m$ with respect to the distribution \eqref{distr_lat_1}. 

\begin{remark}
The randomization procedure introduced in Theorem \ref{teo3} is somehow reminiscent of the definition of the class of Gibbs-type sampling models introduced in \citet{Gne(06)}. This class is defined from the EP-SM with $\alpha<0$ and $\theta=-m\alpha$, for some $m\in\mathbb{N}$, and then it assume that the parameter $m$ is distributed according to an arbitrary distribution on $\mathbb{N}$. See Theorem 12 of \citet{Gne(06)}, and \citet{Gne(10)} for an example. However, differently from the definition of \citet{Gne(06)}, in our context the distribution on $m$ depends on the sample size $n$.
\end{remark}

For $\alpha\in(0,1)$ and $\theta>-\alpha$, \citet{Pit(06)} first investigated the large $n$ asymptotic behaviour of $K_{n}(\alpha,\theta)$. See also \citet{Gne(06)}, and references therein. Let $\stackrel{a.s.}{\longrightarrow}$ denote the almost sure convergence for random variables, and let $S_{\alpha,\theta}$ be the scaled Mittag-Leffler random variable defined above. Theorem 3.8 of \citet{Pit(06)} exploited a martingale convergence argument to show that
\begin{equation}\label{limit_k}
\frac{K_{n}(\alpha,\theta)}{n^{\alpha}}\stackrel{a.s.}{\longrightarrow}S_{\alpha,\theta}
\end{equation}
as $n\rightarrow+\infty$. The random variable $S_{\alpha,\theta}$ is typically referred to as Pitman's $\alpha$-diversity. For $\alpha<0$ and $\theta=-m\alpha$, for some $m\in\mathbb{N}$, the large $n$ asymptotic behaviour of $K_{n}(\alpha,\theta)$ is trivial, that is it holds
\begin{equation}\label{limit_k_neg}
K_{n}(\alpha,\theta)\stackrel{w}{\longrightarrow}m
\end{equation}
as $n\rightarrow+\infty$. See \citet{DF(20a),DF(20b)} for Berry-Esseen type refinements of the large $n$ asymptotic behaviour \eqref{limit_k}, and to \citet{Fav(09),Fav(12)} and \citet{Fav(15)} for generalizations of \eqref{limit_k} with applications to Bayesian nonparametric inference for species sampling problems. See also \citet[Chapter 4]{Pit(06)} for a general treatment of \eqref{limit_k}. According to Theorem \ref{teo3}, it is natural to ask weather there exists an interplay between Theorem \ref{teo1} and the large $n$ asymptotic behaviours \eqref{limit_k} and \eqref{limit_k_neg}.  Hereafter, we show that: i) \eqref{limit_k}, with the almost sure convergence replaced by the convergence in distribution, arises by combining \eqref{cp_k_pos} with i) of Theorem \ref{teo3}; ii) \eqref{cp_k_neg} arises by combining \eqref{limit_k_neg} with ii) of Theorem \ref{teo3}. This provides with an alternative proof of Pitman's $\alpha$-diversity.

\begin{theorem}\label{teo4} 
Let $K_{n}(\alpha,\theta)$ and $K(\alpha,z,n)$ under the EP-SM and the NB-CPSM, respectively. As $n\rightarrow+\infty$
\begin{itemize}
\item[i)] for $\alpha\in(0,1)$ and $\theta>-\alpha$
\begin{equation}\label{limit_k1}
\frac{K_{n}(\alpha,\theta)}{n^{\alpha}}\stackrel{w}{\longrightarrow}S_{\alpha,\theta}.
\end{equation}
\item[ii)] for $\alpha<0$ and $z<0$
\begin{equation}\label{k_negneg}
\frac{K(\alpha,z,n)}{n^{\frac{-\alpha}{1-\alpha}}}\stackrel{w}{\longrightarrow} \frac{(\alpha z)^{\frac{1}{1-\alpha}}}{-\alpha}.
\end{equation}
\end{itemize}
\end{theorem}

\begin{proof}
We show that \eqref{limit_k1} arises by combining \eqref{cp_k_pos} with statement i) of Theorem \ref{teo3}. For any pair of $\mathbb{N}$-valued random variables $U$ and $V$, let $\mathrm{d}_{TV}(U;V)$ be the total variation distance between the distribution of the random variale $U$ and the distribution of the random variable $V$. Also, let $P_{c}$ denote a Poisson random variable with parameter $c>0$. For any $\alpha\in(0,1)$ and $t>0$, we show that as $n\rightarrow+\infty$
\begin{equation} \label{dTV}
\mathrm{d}_{TV}(K(\alpha,tn^{\alpha},n); 1 + P_{tn^{\alpha}}) \rightarrow 0.
\end{equation}
This implies \eqref{limit_k1}. The proof of \eqref{dTV} requires a careful analysis of the probability generating function of $K(\alpha,tn^{\alpha},n)$. In particular, let us define $\omega(t; n, \alpha) := tn^{\alpha} + \frac{t M'_{\alpha}(t)}{M_{\alpha}(t)}$, where $M_{\alpha}(t) := \frac{1}{\pi} \sum_{m=1}^{\infty} \frac{(-t)^{m-1}}{(m-1)!} \Gamma(\alpha m) \sin(\pi\alpha m)$ is the Wright-Mainardi function \citep{Mai(10)}. Then, we apply Corollary 2 of \citet{DF(20a)} in order to conclude that $\mathrm{d}_{TV}(K(\alpha,tn^{\alpha},n); 1 + P_{\omega(t; n, \alpha)}) \rightarrow 0$ as $n\rightarrow+\infty$. Finally, we apply the inequality (2.2) in \citet{Adell(06)} to obtain
\begin{align*}
\mathrm{d}_{TV}(1 + P_{tn^{\alpha}}; 1 + P_{\omega(t; n, \alpha)}) &= \mathrm{d}_{TV}(P_{tn^{\alpha}}; P_{\omega(t; n, \alpha)})\leq \frac{t M'_{\alpha}(t)}{M_{\alpha}(t)} 
\min\left\{1, \frac{\sqrt{(2/e)}}{ \sqrt{\omega(t; n, \alpha)} + \sqrt{tn^{\alpha}} } \right\}
\end{align*}
so that $\mathrm{d}_{TV}(1 + P_{tn^{\alpha}}; 1 + P_{\omega(t; n, \alpha)}) \rightarrow 0$ as $n \rightarrow +\infty$, and \eqref{dTV} follows. Now, keeping $\alpha$ and $t$ fixed as above, we show that \eqref{dTV} entails \eqref{limit_k1}. To this aim, here we introduce the Kolmogorov distance $\mathrm{d}_{K}$ which, for any pair of $\mathbb{R}_{+}$-valued random variables $U$ and $V$, is defined by $\mathrm{d}_{K}(U;V) := 
\sup_{x \geq 0} |\text{Pr}[U\leq x] - \text{Pr}[V\leq x]|$. In particular, the claim that has to be proved is equivalent to the following claim
\begin{displaymath}
\mathrm{d}_{K}(K_{n}(\alpha,\theta)/n^{\alpha}; S_{\alpha,\theta}) \rightarrow 0
\end{displaymath}
as $n\rightarrow+\infty$. In particular, we exploit statement i) of Theorem \ref{teo3}. This leads to the distributional identity $K_{n}(\alpha,\theta) \stackrel{d}{=}K(\alpha,\bar{X}_{\alpha,\theta,n},n)$. Thus, in view of the basic properties of the Kolmogorov distance,
\begin{align} \label{triangular}
\mathrm{d}_{K}(K_{n}(\alpha,\theta)/n^{\alpha}; S_{\alpha,\theta})&\leq \mathrm{d}_{K}(K_{n}(\alpha,\theta); K(\alpha, n^{\alpha} S_{\alpha,\theta},n))\\
&\quad\quad + \mathrm{d}_{K}(K(\alpha, n^{\alpha} S_{\alpha,\theta},n); 1 + P_{n^{\alpha}S_{\alpha,\theta}}) \nonumber \\
&\notag\quad\quad\quad+ \mathrm{d}_{K}([1 + P_{n^{\alpha}S_{\alpha,\theta}}]/n^{\alpha}; S_{\alpha,\theta}), 
\end{align}
where the $\{P_{\lambda}\}_{\lambda \geq 0}$ is thought of here as a homogeneous Poisson process with rate 1, independent of $S_{\alpha,\theta}$. The desired conclusion will be reached as soon
as we will prove that all the three summands on the right-hand side of \eqref{triangular} go to zero as $n \rightarrow +\infty$. Before proceeding, we recall that $\mathrm{d}_{K}(U;V) \leq \mathrm{d}_{TV}(U;V)$. Therefore,
\begin{align*}
&\mathrm{d}_{K}(K_{n}(\alpha,\theta); K(\alpha, n^{\alpha} S_{\alpha,\theta},n)) \\
&\leq \frac 12 \sum_{k=1}^n \Big| \mathscr{C}(n,k;\alpha) \frac{\Gamma(k + \theta/\alpha)}{\alpha\Gamma(\theta/\alpha+1)}
\frac{\Gamma(\theta+1)}{\Gamma(n+\theta)} - \int_0^{+\infty} \frac{\mathscr{C}(n,k;\alpha) (tn^{\alpha})^k}{d_n(t)} f_{S_{\alpha,\theta}}(t) \mathrm{d}t \Big|
\end{align*}
with $d_n(t) := \sum_{j=1}^n \mathscr{C}(n,j;\alpha) (tn^{\alpha})^j$. Now, let us define the following quantity: $d^{\ast}_n(t) := e^{tn^{\alpha}} (n-1)! \frac{1}{t^{1/\alpha}} f_{\alpha}(\frac{1}{t^{1/\alpha}})$. Accordingly, we can majorize the above right-hand side by means of the following quantity
\begin{align*}
&\frac 12 \sum_{k=1}^n \Big| \mathscr{C}(n,k;\alpha) \frac{\Gamma(k + \theta/\alpha)}{\alpha\Gamma(\theta/\alpha+1)}
\frac{\Gamma(\theta+1)}{\Gamma(n+\theta)} - \int_0^{+\infty} \frac{\mathscr{C}(n,k;\alpha) (tn^{\alpha})^k}{d^{\ast}_n(t)} f_{S_{\alpha,\theta}}(t) \mathrm{d}t \Big| \\
&+\frac 12\int_0^{+\infty} \frac{|d^{\ast}_n(t) - d_n(t)|}{d^{\ast}_n(t)}f_{S_{\alpha,\theta}}(t) \mathrm{d}t \ . 
\end{align*}
Accordingly, by exploiting the identity $\int_0^{+\infty} \frac{(tn^{\alpha})^k}{d^{\ast}_n(t)}f_{S_{\alpha,\theta}}(t) \mathrm{d}t = \frac{1}{(n-1)!} \frac{\Gamma(k + \theta/\alpha)}{n^{\theta}}
\frac{\Gamma(\theta+1)}{\alpha\Gamma(\theta/\alpha +1)}$, we can write that
\begin{align*}
&\sum_{k=1}^n \Big| \mathscr{C}(n,k;\alpha) \frac{\Gamma(k + \theta/\alpha)}{\alpha\Gamma(\theta/\alpha+1)}
\frac{\Gamma(\theta+1)}{\Gamma(n+\theta)} - \int_0^{+\infty} \frac{\mathscr{C}(n,k;\alpha) (tn^{\alpha})^k}{d^{\ast}_n(t)} f_{S_{\alpha,\theta}}(t) \mathrm{d}t \Big| =\Big| 1 - \frac{\Gamma(n+\theta)}{\Gamma(n)n^{\theta}}\Big|
\end{align*}
which goes to zero as $n \rightarrow +\infty$ for any $\theta > -\alpha$, by a direct application of Stirling's approximation. To show that the integral $\int_0^{+\infty} \frac{|d^{\ast}_n(t) - d_n(t)|}{d^{\ast}_n(t)}f_{S_{\alpha,\theta}}(t) \mathrm{d}t$ also goes to zero as $n \rightarrow +\infty$, we may resort to the identities (13)--(14) of \citet{DF(20a)}, as well as Lemma 3 in \citet{DF(20a)}. In particular, let $\Delta\,:\,(0,+\infty)\rightarrow(0,+\infty)$ denote a suitable continuous function independent of $n$, and such that $\Delta(z)=O(1)$ as $z\rightarrow0$ and $\Delta(z)f_{\alpha}(1/z)=O(z^{-\infty})$ as $z\rightarrow+\infty$. Then, we write that
\begin{align*}
&\int_0^{+\infty} \frac{|d^{\ast}_n(t) - d_n(t)|}{d^{\ast}_n(t)}f_{S_{\alpha,\theta}}(t) \mathrm{d}t \\
& \leq \Big| \frac{(n/e)^n \sqrt{2\pi n}}{n!} - 1\Big| + \left(\frac{(n/e)^n \sqrt{2\pi n}}{n!}\right) \frac{1}{n} \int_0^{+\infty} \Delta(t^{1/\alpha}) f_{S_{\alpha,\theta}}(t) \mathrm{d}t\ .
\end{align*}
Since we have that $\int_0^{+\infty} \Delta(t^{1/\alpha}) f_{S_{\alpha,\theta}}(t) \mathrm{d}t < +\infty$ by Lemma 3 of \citet{DF(20a)}, both the summands on the above right-hand side go to zero as $n \rightarrow +\infty$, again by a direct application of Stirling's approximation. Thus, the first summand on the right-hand side of \eqref{triangular} goes to zero as $n \rightarrow +\infty$. As for the second summand on the right-hand side of \eqref{triangular}, it can be bounded by the quantity
\begin{displaymath}
\int_0^{+\infty} \mathrm{d}_{TV}(K(\alpha,tn^{\alpha},n); 1 + P_{tn^{\alpha}}) f_{S_{\alpha,\theta}}(t) \mathrm{d}t\ .
\end{displaymath}
By a dominated convergence argument, this quantity goes to zero as $n \rightarrow +\infty$ as a consequence of \eqref{dTV}. Finally, for the third summand on the right-hand side of \eqref{triangular}, we can resort to a conditioning argument in order to reduce the problem to a direct application of the well-known law of large numbers for renewal processes. See, e.g. \cite[Section 10.2]{GS(06)} and references therein. In particular, this leads to $n^{-\alpha} P_{tn^{\alpha}} \stackrel{a.s.}{\longrightarrow} t$ for any $t > 0$, which entails that 
$n^{-\alpha} P_{n^{\alpha}S_{\alpha,\theta}} \stackrel{a.s.}{\longrightarrow} S_{\alpha,\theta}$ as $n \rightarrow +\infty$.
Thus, this third term also goes to zero as $n \rightarrow +\infty$, and \eqref{limit_k1} follows.

Now, we consider \eqref{k_negneg}, showing that it arises by combining \eqref{limit_k_neg} with statement ii) of Theorem \ref{teo3}. In particular, by an obvious conditioning argument, we can write that as $n\rightarrow+\infty$ it holds true that 
\begin{displaymath}
\frac{K_{n}(\alpha,\tilde{X}_{\alpha,z,n}|\alpha|)}{\tilde{X}_{\alpha,z,n}} \stackrel{a.s.}{\longrightarrow} 1.
\end{displaymath}
At this stage, we consider the probability generating function of the random variable $\tilde{X}_{\alpha,z,n}$, and therefore we immediately obtain $\mathbb{E}[s^{\tilde{X}_{\alpha,z,n}}] := B_n(-sz)/B_n(-z)$ for $n \in \mathbb{N}$ and $s \in [0,1]$
with the same $B_n$ as in \eqref{bernardoNEW} and \eqref{BnWright}. Therefore, the asymptotic expansion we already provided in \eqref{BnAsymptotic} entails 
\begin{equation}\label{eq_teo4}
\frac{\tilde{X}_{\alpha,z,n}}{n^{\frac{-\alpha}{1-\alpha}}}  \stackrel{w}{\longrightarrow} \frac{(\alpha z)^{\frac{1}{1-\alpha}}}{-\alpha}
\end{equation}
as $n\rightarrow+\infty$. In particular, \eqref{eq_teo4} follows by applying exactly the same arguments used to prove \eqref{cp_k_neg}. Now, since
\begin{displaymath}
\frac{K_{n}(\alpha,\tilde{X}_{\alpha,z,n}|\alpha|)}{n^{\frac{-\alpha}{1-\alpha}}} \stackrel{d}{=} \frac{K_{n}(\alpha,\tilde{X}_{\alpha,z,n}|\alpha|)}{\tilde{X}_{\alpha,z,n}} \frac{\tilde{X}_{\alpha,z,n}}{n^{\frac{-\alpha}{1-\alpha}}} ,
\end{displaymath}
the claim follows from a direct application of well-know Slutsky's theorem. This completes the proof.
\end{proof}

%%%%%%%%%%%%%%%%%%%%%%%%%%%%%%%%
%%%%%%%%%%%%%%%%%%%%%%%%%%%%%%%%
%%%%%%%%%%%%%%%%%%%%%%%%%%%%%%%%
%%%%%%%%%%%%%%%%%%%%%%%%%%%%%%%%

\section{Discussion} The NB-CPSM is a compound Poisson sampling model generalizing the popular LS-CMSM. In this paper, we introduced a compound Poisson perspective of the EP-SM in terms of the NB-CPSM, thus extending the well-known compound Poisson perspective of the E-SM in terms of the LS-CPSM. We conjecture that an analogous perspective holds true for the class of $\alpha$-stable Poisson-Kingman sampling models \citep{Pit(03),Pit(06)}, of which the EP-SM is a noteworthy special case. That is, for $\alpha\in(0,1)$, we conjecture that an $\alpha$-stable Poisson-Kingman sampling model admits a representation as a randomized NB-CPSM with $\alpha\in(0,1)$ and $z>0$, where the randomization acts on $z$ with respect a scale mixture between a Gamma and a suitable transformation of the Mittag-Leffler distribution. We believe that such a compound Poisson representation would be critical in order to introduce Berry-Esseen type refinements of the large $n$ asymptotic behaviour of $K_{n}$ under $\alpha$-stable Poisson-Kingman sampling models. See \citet[Section 6.1]{Pit(03)}, and references therein. Such a line of research aims at extending preliminary works of \citet{DF(20a),DF(20b)} on Berry-Esseen type theorems  under the EP-SM. Work on this, and on the more general settings induced by normalized random measures \citep{Reg(03)} and Poisson-Kingman models \citep{Pit(03)}, is ongoing.

\section*{Acknowledgement}

Emanuele Dolera and Favaro received funding from the European Research Council (ERC) under the European Union's Horizon 2020 research and innovation programme under grant agreement No 817257. Emanuele Dolera and Stefano Favaro gratefully acknowledge the financial support from the Italian Ministry of Education, University and Research (MIUR), ``Dipartimenti di Eccellenza" grant 2018-2022.

%%%%%%%%%%%%%%%%%%%%%%%%%%%%%%%%
%%%%%%%%%%%%%%%%%%%%%%%%%%%%%%%%
%%%%%%%%%%%%%%%%%%%%%%%%%%%%%%%%
%%%%%%%%%%%%%%%%%%%%%%%%%%%%%%%%

%%%%%%%%%%%%%%%%%%%%%%%%%%%%%%%%
%%%%%%%%%%%%%%%%%%%%%%%%%%%%%%%%
%%%%%%%%%%%%%%%%%%%%%%%%%%%%%%%%
%%%%%%%%%%%%%%%%%%%%%%%%%%%%%%%%

\end{document}